\documentclass[11pt,a4paper]{amsart}

\usepackage{mathtools,amssymb,amsthm}
\usepackage[l2tabu,orthodox]{nag} 
\usepackage[all, warning]{onlyamsmath} 
\usepackage{bm}
\usepackage{mathrsfs} 
\usepackage{stmaryrd} 
\usepackage{url}
\usepackage{mymacro}

\usepackage{xcolor}

\definecolor{darkgreen}{rgb}{0.0, 0.6, 0.0}

\usepackage{tikz}
\usetikzlibrary{intersections, calc, arrows, cd}

\usepackage[colorlinks,citecolor=darkgreen,linkcolor=red]{hyperref}
\RequirePackage{cleveref}

\usepackage{geometry}
\numberwithin{equation}{section}
\numberwithin{figure}{section}
\numberwithin{table}{section}
\allowdisplaybreaks 

\usepackage{enumitem}
\setenumerate[1]{label={\upshape(\arabic*)},nosep}
\setenumerate[2]{label={\upshape(\alph*)},nosep}
\setitemize{nosep}
\setdescription{nosep}

\newlist{enumarabic}{enumerate}{1}
\setlist*[enumarabic]{label={\upshape(\arabic*)}, nosep}

\newlist{enumAlph}{enumerate}{1}
\setlist*[enumAlph]{label={\upshape(\Alph*)}, nosep}

\newlist{enumalph}{enumerate}{1}
\setlist*[enumalph]{label={\upshape(\alph*)}, nosep}

\newlist{enumRoman}{enumerate}{1}
\setlist*[enumRoman]{label={\upshape(\Roman*)}), nosep}

\newlist{enumroman}{enumerate}{1}
\setlist*[enumroman]{label={\upshape(\roman*)}, nosep}

\usepackage{cleveref}
 \crefname{section}{\S\!}{\S\S\!}
 \crefname{subsection}{\S\!}{\S\S\!}
 \crefname{subsubsection}{\S\!}{\S\S\!}
 \crefname{equation}{equation}{equations}

\theoremstyle{plain}
\newtheorem{theorem}{Theorem}[section]
 \crefname{theorem}{Theorem}{Theorems}
\newtheorem{theorem-definition}[theorem]{Theorem-Definition}
 \crefname{theorem-definition}{Theorem-Definition}{Theorem-Definition}
\newtheorem{lemma}[theorem]{Lemma}
 \crefname{lemma}{Lemma}{Lemmas}
\newtheorem{proposition}[theorem]{Proposition}
 \crefname{proposition}{Proposition}{Propositions}
\newtheorem{proposition-definition}[theorem]{Proposition-Definition}
 \crefname{proposition-definition}{Proposition-Definition}{Proposition-Definition}
\newtheorem{corollary}[theorem]{Corollary}
 \crefname{corollary}{Corollary}{Corollaries}

 \crefname{claim}{Claim}{Claims}

 \crefname{question}{Question}{Questions}

 \crefname{problem}{Problem}{Problems}

 \crefname{conjecture}{Conjecture}{Conjectures}
\newtheorem{theoremA}{Theorem}
 \crefname{theoremA}{Theorem}{Theorems} 

\theoremstyle{definition}
\newtheorem{definition}[theorem]{Definition}
 \crefname{definition}{Definition}{Definitions}

 \crefname{fact}{Fact}{Facts}

 \crefname{notation}{Notation}{Notations}

 \crefname{observation}{Observation}{Observations}
\newtheorem{remark}[theorem]{Remark}
 \crefname{remark}{Remark}{Remarks}
\newtheorem{example}[theorem]{Example}
 \crefname{example}{Example}{Examples}
\newtheorem{construction}[theorem]{Construction}
 \crefname{construction}{Construction}{Constructions}

\theoremstyle{remark}

\def\Itor{\textrm{$I$-}\mathrm{tor}}
\def\Iztor{\textrm{$I_0$-}\mathrm{tor}}
\def\Izstor{\textrm{$I_0^{s.\flat}$-}\mathrm{tor}}

\begin{document}
\title[Regular rings and perfectoid towers]{Regular rings and perfectoid towers}

\author[K.~Hayashi]{Kazuki Hayashi}
\address{Department of Mathematics, Institute of Science Tokyo, 2-12-1 Ookayama, Meguro, Tokyo 152-8551}
\email{hazuki0694@gmail.com}




\begin{abstract}
We prove a mixed-characteristic analogue of Kunz's theorem in terms of perfectoid towers: a Noetherian local ring of residue characteristic $p$ is regular if and only if it admits a flat map to a Noetherian ring that extends to a perfectoid tower. This result is deduced from another mixed-characteristic analogue due to O.~Gabber and J.~Lurie. We also characterize regularity for perfectoid towers via vanishing of single higher $\Tor$-module of the residue field with a perfectoid algebra.
\end{abstract}

\maketitle

\setcounter{tocdepth}{2}
\tableofcontents
\section{Introduction}

This paper investigates the regularity of commutative Noetherian local rings of positive/mixed characteristic via perfectoid towers. A commutative $\F_p$-algebra $R$ is called \emph{perfect} if its absolute Frobenius $\varphi\colon R\to R$ is bijective. The injectivity of $\varphi$ amounts to saying that $R$ is reduced, and in this case, the tower
\begin{equation}
\label{eq:perfecttower}
R\xr{\varphi}R\xr{\varphi}R\xr{\varphi}\cdots
\end{equation}
is called a \emph{perfect tower}. The inductive limit $R_{\mathrm{perf}}\coloneqq \varinjlim_\varphi R$ of \eqref{eq:perfecttower} is a perfect $\F_p$-algebra, and the notion of perfect towers provides a tower-theoretic analogue of perfect $\F_p$-algebras. On the other hand, \emph{perfectoid rings} are generalizations of perfect $\F_p$-algebras to mixed characteristic. Recently, S.~Ishiro, K.~Nakazato, and K.~Shimomoto \cite{INS25} introduced \emph{perfectoid towers} as a generalization of perfect towers (\cref{def:perfectoidtower}):
\begin{equation}
\label{eq:perfectoidtower}
R=R_0 \xr{t_0} R_1 \xr{t_1} R_2 \xr{t_2} \cdots.
\end{equation}
The $p$-adic completion $\wh{R_\infty}$ of the inductive limit $R_\infty\coloneqq\varinjlim_{i\geq 0}R_i$ of \eqref{eq:perfectoidtower} is a perfectoid ring, and the notion of perfectoid towers provides a tower-theoretic analogue of perfectoid rings.

In the perfectoid tower associated to a complete regular local rings (see \cref{ex:perfectoidtower} (3)), the transition maps $t_i\colon R_i\to R_{i+1}$ are flat for all $i\geq 0$. We first show that the same holds for \emph{general} perfectoid towers arising from regular rings; more precisely, the flatness of the $i$-th transition map characterizes the regularity of the $i$-th layer $R_i$ (\cref{thm:regularFlat}). 
This result is deduced from a mixed characteristic variant of Kunz's theorem due to J.~Lurie, suggested by O.~Gabber (see \cref{thm:LurKunz}). It is also essential that if some transition map is flat, then so are all transition maps (\cref{cor:flatrigid}).

As a consequence, we obtain the following criterion of regular rings via perfectoid towers.

\begin{theoremA}[{\cref{thm:towerKunz}}]
\label{thm:A}
Let $R$ be a commutative Noetherian local ring of residue characteristic $p$. Then the following conditions are equivalent.
  \begin{enumerate}
  \item The local ring $R$ is regular.
  \item There exists a perfectoid tower \eqref{eq:perfectoidtower} consisting of Noetherian rings such that for any \emph{(}or, equivalently, some\emph{)} $i\geq 0$ the transition map $t_i\colon R_i\to R_{i+1}$ is flat.
  \end{enumerate}
\end{theoremA}

Although this theorem---except for the parenthetical assertion---can be deduced from the $p$-adic Kunz theorem of B.~Bhatt, S.~B.~Iyengar, and L.~Ma (see \cref{rem:anotherpf}), we present an independent proof. We emphasize that our proof does not rely on the Auslander--Buchsbaum--Serre regularity criterion. Thus we can apply \cref{thm:A} to obtain a new proof of the stability of regularity under localization at prime ideals containing $p$ (see \cref{prop:RLRloc}).

On the other hand, Bhatt--Iyengar--Ma gave a sharper characterization of regularity via perfect(oid) algebras. For instance, if $(R,\m,k)$ is an excellent reduced Noetherian local $\F_p$-algebra, it suffices to assume the vanishing of a \emph{single} $\Tor$-module $\Tor^R_q(R_{\mathrm{perf}},k)$ (\cite[Theorem 4.13 (1)]{BIM19}).\footnote{This result was previously proved by I.~M.~Aberbach and J.~Li (\cite[Remark 3.6]{AL08}) using different methods, under the mild additional assumption that $R$ is equidimensional.} A key input in their proof is that systems of parameters for $R$ are \emph{weakly proregular} on $R_{\mathrm{perf}}$. In this paper, we establish such a rigidity result for more general perfect(oid) algebras:

\begin{theoremA}[{\cref{thm:WPRreg}}]
\label{thm:B}
Let $(R,\m,k)$ be a commutative Noetherian local ring. Then $R$ is regular if any one of the following conditions is satisfied.
  \begin{enumerate}
  \item The local ring $R$ has positive characteristic, and there exists a perfect $R$-algebra $A$ such that $\Tor^R_q(A,k)=0$ for some $q\geq 1$.
  \item The local ring $R$ has mixed characteristic and there exists a perfectoid tower \eqref{eq:perfectoidtower} arising from a pair $(R,I_0)$ such that $R$ is $I_0$-torsion free, together with a perfectoid $\wh{R_\infty}$-algebra $A$ such that $\Tor^R_q(A,k)=0$ for some $q\geq 1$.
  \end{enumerate}
\end{theoremA}

A key input in its proof is a result of of Gabber and L.~Ramero on weakly proregular sequences in perfectoid rings (see \cref{prop:WPR_GR}), together with the author's previous work \cite{HIS26}. Finally, we give two constructions of perfectoid towers arising from the pair consisting of regular rings and ideals not generated by the prime number $p$ (\cref{prop:Rees,prop:g}). The first construction is obtained by taking the blow-up of an unramified regular local ring at its closed point.

\addtocontents{toc}{\protect\setcounter{tocdepth}{-2}} 

\subsection*{Acknowledgements} 
The author would like to express his sincere gratitude to Ryo Ishizuka for suggesting the idea of \cref{prop:g}. The author also thanks Shinnosuke Ishiro, Kei Nakazato, Kazuma Shimomoto, and Tatsuki Yamaguchi for their continuous encouragement.

\subsection*{Notations and conventions}
  \begin{itemize}
  \item We consistently fix a prime number $p>0$.
  \item All rings are assumed to be commutative and unital (unless otherwise stated).
  \item For an $\F_p$-algebra $R$, let $\varphi=\varphi_R\colon R\to R$ denote the absolute Frobenius.
  \item By a \emph{pair} we simply mean a couple $(A,I)$ consisting of a ring $A$ and an ideal $I$ of $A$. When the ideal $I$ is principal, say $I=(a)$, then we often write $(A,a)$ in place of $(A,I)$.
  \item For a pair $(A,I)$ and an $A$-module $M$, we say that an element $x\in M$ is \emph{$I$-torsion} if for all $a\in I$ there exists an integer $n>0$ such that $a^nx=0$. Let $M_{\textrm{$I$-}\mathrm{tor}}$ denote the $A$-submodule of $M$ consisting of all $I$-torsion elements in $M$. We say that $M$ is \emph{$I$-torsion free} if $M_{\textrm{$I$-}\mathrm{tor}}=0$. Note that we follow the terminologies in \cite{FKI}.
  \item For a pair $(A,I)$, when we say an $A$-module $M$ is $I$-adically complete, we always mean that $M$ is Hausdorff complete with respect to the $I$-adic topology.
  \end{itemize}

\addtocontents{toc}{\protect\setcounter{tocdepth}{2}} 
\section{Recollections on perfectoid towers}

In this section, we recall the definition of perfectoid towers with properties and examples most relevant to us. 

For an $\F_p$-algebra $R$, let $\varphi=\varphi_R\colon R\to R$ denote the absolute Frobenius: $\varphi(x)=x^p$ for $x\in R$. The ring $R$ is reduced if and only if $\varphi$ is injective. In this case, we have a ``tower'' of $\F_p$-algebras
\begin{equation}
\label{eq:perftower1}
R\xr{\varphi} R\xr{\varphi}R\xr{\varphi}\cdots.
\end{equation}

In general, a \emph{tower of rings/$\F_p$-algebras} is an inductive system of rings/$\F_p$-algebras $\textrm{{\boldmath $R$}}=\{R_i\}_{i\geq 0}=\{R_i,t_i\}_{i\geq 0}$, and a \emph{morphism of towers} is a morphism of inductive systems. A \emph{perfect tower} is a tower of $\F_p$-algebras that is isomorphic to \eqref{eq:perftower1} for a reduced $\F_p$-algebra $R$.

Perfectoid towers are generalizations of perfect towers to mixed characteristic. The following class of towers is a predecessor of perfectoid towers.

\begin{definition}[{\cite[Definition 3.4]{INS25}}]
\label{def:inseptower}
Let $(R,I)$ be a pair, and $\textrm{{\boldmath $R$}}=\{R_i,t_i\}_{i\geq 0}$ a tower of rings.
We say that {\boldmath $R$} is a \emph{purely inseparable tower arising from $(R,I)$} if it satisfies the following conditions.
    \begin{itemize}
    \item[\textbf{(a)}] $R_0=R$ and $p\in I$.
    \item[\textbf{(b)}] For any $i\geq 0$, the ring map $\ol{t_i}\colon R_i/IR_i \to R_{i+1}/IR_{i+1}$ induced by $t_i$ is injective.
    \item[\textbf{(c)}] For any $i\geq 0$, we have $\Im(\varphi_{R_{i+1}/IR_{i+1}}) \subset \Im(\ol{t_i})$.
    \end{itemize}
Under these assumptions, for any $i\geq 0$ the absolute Frobenius $\varphi\colon R_{i+1}/IR_{i+1}\to R_{i+1}/IR_{i+1}$ factors uniquely through $\ol{t_i}$ as follows:
\[
\begin{tikzcd}
R_{i+1}/IR_{i+1} \rar["\varphi"] \ar[rd,"F_i"'] & R_{i+1}/IR_{i+1} \\
 & R_i/IR_i \uar["\ol{t_i}"',hookrightarrow]
\end{tikzcd}
\]
We call $F_i$ the \emph{$i$-th Frobenius projection} (of {\boldmath $R$} associated to $(R,I)$).
\end{definition}

By definition, the extension of $\F_p$-algebras $\ol{t_i}\colon R_i/IR_i\hookrightarrow R_{i+1}/IR_{i+1}$ is purely inseparable in the following sense:

\begin{definition}[{\cite[(6.15) Discussion]{HH94}}]
We say that an extension of $\F_p$-algebras $A\subset B$ is \emph{purely inseparable} if every element of $B$ has some $p$-power in $A$.
\end{definition}

This definition can be extended to any homomorphism of $\F_p$-algebras $f\colon A\to B$ in a natural way: we say that $f\colon A\to B$ is \emph{purely inseparable} if the extension $\Im(f)\subset B$ is purely inseparable.

\begin{remark}
\label{rem:purelyinsep}
Let $f\colon A\to B$ be a purely inseparable homomorphism of $\F_p$-algebras.
  \begin{enumerate}
  \item The ring homomorphism $f$ is integral, that is, the extension $\Im(f)\subset B$ is integral.
  \item For any homomorphism $A\to A'$ of $\F_p$-algebras, the induced homomorphism $A'\to B\otimes_AA'$ is purely inseparable.
  \end{enumerate}
\end{remark}

The notion of purely inseparable homomorphisms is related to that of universal homeomorphisms as follows:

\begin{proposition}
\label{prop:piuh}
Let $f\colon A\to B$ be a homomorphism of $\F_p$-algebras. Then the following conditions are equivalent.
  \begin{enumerate}
  \item $f$ is a universal homeomorphism, that is, $\Spec(B)\to\Spec(A)$ is a universal homeomorphism in the category of schemes.
  \item $\Ker(f)$ is contained in the nilradical of $A$, and $f$ is purely inseparable.
  \end{enumerate}
\end{proposition}

\begin{proof}
The assertion follows from \cite[Proposition 3.3.3]{Alp14}.
But, for the reader's convenience, let us include here the proof of the implication ``(2) $\Rightarrow$ (1).''
Since the property ``purely inseparable'' is stable under base change (\cref{rem:purelyinsep} (2)), it suffices to show that $\Spec B\to\Spec A$ is a homeomorphism. By assumption, we may assume that $f$ is injective. Then $A\subset B$ is an integral extension (\cref{rem:purelyinsep} (1)), and thus $\Spec(B)\to\Spec(A)$ is surjective and closed. It remains to show the injectivity. 
Assume $\frakP,\frakP' \in\Spec(B)$ satisfy $\frakP\cap A=\frakP'\cap A\eqqcolon\p$. Pick $x\in\frakP$. By assumption, $x^n\in A$ for some $n>0$. Then $x^n\in\frakP\cap A=\p\subset\frakP'$, and so $x\in\frakP'$. Hence $\frakP\subset\frakP'$. Similarly, we have $\frakP'\subset\frakP$.
\end{proof}

Consequently, we obtain the following result on purely inseparable towers.

\begin{corollary}
\label{cor:piuh}
Let $\textrm{{\boldmath $R$}}=\{R_i,t_i\}_{i\geq 0}$ be a purely inseparable tower arising from a pair $(R,I)$. Then for any $i\geq 0$, the following assertions hold.
  \begin{enumerate}
  \item The map $\ol{t_i}\colon R_i/IR_i\to R_{i+1}/IR_{i+1}$ induced by $t_i$ is a universal homeomorphism.
  \item The $i$-th Frobenius projection $F_i\colon R_{i+1}/IR_{i+1}\to R_i/IR_i$ is a universal homeomorphism.
  \item If both $R_i$ and $R_{i+1}$ are $I$-adically henseilan, then $t_i\colon R_i\to R_{i+1}$ induces an equivalence of the categories of finite \'{e}tale algebras
  \[
  \Fet_{/R_i} \xr{\simeq} \Fet_{/R_{i+1}}.
  \]
  \end{enumerate}
\end{corollary}

\begin{proof}
(1) follows from \cref{prop:piuh}.

(2) The absolute Frobenius of $R_{i+1}/IR_{i+1}$ is a universal homeomorphism (\cite[\href{https://stacks.math.columbia.edu/tag/0CC8}{Tag 0CC8}]{stacks-project}). Then so is $F_i$ by (1) and the 2-out-of-3 property (\cite[\href{https://stacks.math.columbia.edu/tag/0H2M}{Tag 0H2M}]{stacks-project}).

(3) follows from (1), \cite[\href{https://stacks.math.columbia.edu/tag/09ZL}{Tag 09ZL}]{stacks-project}, and \cite[\href{https://stacks.math.columbia.edu/tag/0BQN}{Tag 0BQN}]{stacks-project}.
\end{proof}

\begin{remark}
\cref{cor:piuh} (3) was already proved in \cite[Lemma 3.12]{INS25}; however, the proof there passes to the direct perfections of $R_i/I_0R_i$ and $R_{i+1}/I_0R_{i+1}$. Our proof focuses only $R_i/I_0R_i$ and $R_{i+1}/I_0R_{i+1}$ itself.
\end{remark}

The following lemma will be used later.

\begin{lemma}
\label{lem:localtower}
Let $\textrm{{\boldmath $R$}}=\{R_i,t_i\}_{i\geq 0}$ be a purely inseparable tower arising from a pair $(R,I_0)$. Fix $i\geq 0$, and suppose that $R_i$ and $R_{i+1}$ are local rings with the maximal ideals $\m_i$ and $\m_{i+1}$, respectively. Suppose furthermore that $\ol{R_i}\coloneqq R_i/IR_i\neq 0$ and $\ol{R_{i+1}}\coloneqq R_{i+1}/IR_{i+1}\neq 0$.
  \begin{enumerate}
  \item The ring homomorphisms $t_i,\ol{t_i}$, and $F_i$ are local.
  \item Assume that the $i$-th Frobenius projection $F_i\colon\ol{R_{i+1}}\to \ol{R_i}$ is surjective.
    \begin{enumerate}
    \item $\m_i\ol{R_{i+1}}=(\m_{i+1}\ol{R_{i+1}})^{[p]}$, where $(-)^{[p]}$ denotes the Frobenius power.
    \item The closed fiber of $\Spec(R_{i+1})\to\Spec(R_i)$ is of dimension $0$.
    \end{enumerate}
  \end{enumerate}
\end{lemma}

\begin{proof}
(1) follows from \cite[Lemma 3.10 (1)]{INS25}.

(2) (a) It follows from (1) that $F_i^{-1}(\m_i\ol{R_i})=\m_{i+1}\ol{R_{i+1}}$. Since $F_i$ is surjective, we have $\m_i\ol{R_i}=F_i(\m_{i+1}\ol{R_{i+1}})$. Then $\m_i\ol{R_{i+1}}=(\m_{i+1}\ol{R_{i+1}})^{[p]}$, as claimed.

(b) The closed fiber of $\Spec(R_{i+1}) \to \Spec(R_i)$ is the same as that of $\Spec(\ol{R_{i+1}}) \to \Spec(\ol{R_i})$, which is of dimension $0$ by (a).
\end{proof}

Let us recall the definition of (pre)perfectoid towers.

\begin{definition}[{\cite[Definition 3.21]{INS25}, \cite[Definition 2.3]{Ha26a}}]
\label{def:perfectoidtower}
Let $(R,I_0)$ be a pair, and $\textrm{{\boldmath $R$}}=\{R_i,t_i\}_{i\geq 0}$ a tower of rings. We say that {\boldmath $R$} a \emph{perfectoid tower arising from $(R,I_0)$} if it is a purely inseparable tower arising from $(R,I_0)$ satisfying the following additional conditions.
  \begin{itemize}
  \item[\textbf{(d)}] For any $i\geq 0$, the $i$-th Frobenius projection $F_i\colon R_{i+1}/I_0R_{i+1}\to R_i/I_0R_i$ is surjective.
    \item[\textbf{(e)}] For any $i\geq 0$, $R_i$ is $I_0R_i$-adically Zariskian (i.e., $I_0R_i$ is contained in the Jacobson radical of $R_i$). 
  \item[\textbf{(f)}] $I_0$ is a principal ideal, and $R_1$ contains a principal ideal $I_1$ satisfying the following conditions.
    \begin{itemize}
    \item[\textbf{(f-1)}] $I_1^p=I_0R_1$.
    \item[\textbf{(f-2)}] For any $i\geq 0$, $\Ker(F_i)=I_1(R_{i+1}/I_0R_{i+1})$.
    \end{itemize}  
  \item[\textbf{(g)}] For any $i\geq 0$, $I_0(R_i)_{\textrm{$I_0$-}\mathrm{tor}}=(0)$. Moreover, there exists a bijection $(F_i)_{\mathrm{tor}}\colon (R_{i+1})_{\textrm{$I_0$-}\mathrm{tor}}\to (R_i)_{\textrm{$I_0$-}\mathrm{tor}}$ such that the diagram of sets
  \[
  \begin{tikzcd}[column sep=small]
  (R_{i+1})_{\textrm{$I_0$-}\mathrm{tor}} \dar["(F_i)_{\mathrm{tor}}"'] \rar[hookrightarrow] & R_{i+1} \rar[twoheadrightarrow] & R_{i+1}/I_0R_{i+1} \dar["F_i"] \\
  (R_i)_{\textrm{$I_0$-}\mathrm{tor}} \rar[hookrightarrow] & R_i \rar[twoheadrightarrow] & R_i/I_0R_i
  \end{tikzcd}
  \]
  commutes.
  \end{itemize}
A \emph{preperfectoid tower} is only required to satisfy all conditions except \textbf{(e)}.
\end{definition}

Let $\textrm{{\boldmath $R$}}=\{R_i\}_{i\geq 0}$ be a preperfectoid tower arising from a pair $(R,I_0)$. Then there exists a unique sequence of finitely generated ideals $\{I_i\subset R_i\}_{i\geq 2}$ such that $F_i(I_{i+1}(R_{i+1}/I_0R_{i+1}))=I_i(R_i/I_0R_i)$ for all $i\geq 0$. We call $I_i$ the \emph{$i$-th perfectoid pillar}.

Furthermore, to a preperfectoid tower $\textrm{{\boldmath $R$}}=\{R_i,t_i\}_{i\geq 0}$ arising from $(R,I_0)$ we can associate a tower of $\F_p$-algebras $\textrm{{\boldmath $R$}}^\flat=\{R_i^{s.\flat},t_i^{s.\flat}\}_{i\geq 0}$ called the \emph{tilt} of {\boldmath $R$}, as follows.
  \begin{itemize}
  \item For any $i\geq 0$, let
  \[
  R_i^{s.\flat}\stackrel{\mathrm{def}}{=} \varprojlim(\cdots \xr{F_{i+2}} R_{i+2}/I_0R_{i+2} \xr{F_{i+1}} R_{i+1}/I_0R_{i+1} \xr{F_i} R_i/I_0R_i).
  \]
  We call $R_i^{s.\flat}$ the \emph{$i$-th small tilt} (of {\boldmath $R$} associated to $(R,I_0)$).
  \item For any $i\geq 0$, let $t_i^{s.\flat}\colon R_i^{s.\flat}\to R_{i+1}^{s.\flat}$ be the unique ring homomorphism such that the diagram of rings
  \[
  \begin{tikzcd}
  R_i^{s.\flat} \rar["t_i^{s.\flat}"] \dar & R_{i+1}^{s.\flat} \dar \\
  R_{i+m}/I_0R_{i+m} \rar["\ol{t_{i+m}}"] & R_{i+m+1}/I_0R_{i+m+1}
  \end{tikzcd}
  \]
  commutes for all $m\geq 0$, where the vertical arrows are the $m$-th projections.
  \end{itemize}
For each $i\geq 0$, the \emph{small tilt} of the $i$-th perfectoid pillar $I_i$ is the kernel
\[
I_i^{s.\flat}\coloneqq \Ker(R_i^{s.\flat}\xr{\pr_0} R_i/I_0R_i \twoheadrightarrow R_i/I_i).
\]
In particular, the ideal $I_0^{s.\flat}\subset R^{s.\flat}$ is a counterpart of $I_0\subset R$ in the following sense:

\begin{theorem}[{\cite[Lemma 3.39, Proposition 3.41]{INS25}}]
\label{thm:INStilt}
Let $\textrm{{\boldmath $R$}}=\{R_i\}_{i\geq 0}$ be a perfectoid tower arising from a pair $(R,I_0)$.
  \begin{enumerate}
  \item For every $i\geq 0$, the $0$-th projection $R_i^{s.\flat} \to R_i/I_0R_i$ induces an isomorphism of $\F_p$-algebras
  \[
  R_i^{s.\flat}/I_0^{s.\flat}R_i^{s.\flat} \xr{\cong} R_i/I_0R_i.
  \]
  \item The tilt $\textrm{{\boldmath $R$}}^\flat=\{R_i^{s.\flat}\}_{i\geq 0}$ is a perfectoid tower arising from $(R^{s.\flat},I_0^{s.\flat})$.
  \end{enumerate}
\end{theorem}

One can also show that $\textrm{{\boldmath $R$}}^\flat$ is a perfect tower (\cite[Proposition 3.10 (2)]{INS25}).

\begin{proposition}
\label{prop:towerLocal}
Let $\textrm{{\boldmath $R$}}=\{R_i\}_{i\geq 0}$ be a perfectoid tower arising from a pair $(R,I_0)$.
  \begin{enumerate}
  \item For any $i\geq 0$, $R_i$ is local if and only if so is $R_{i+1}$. In particular, if $R_i$ is local for some $i\geq 0$, then $R_i$ is local for any $i\geq 0$.
  \item For any $i\geq 0$, $R_i$ is local if and only if so is $R_i^{s.\flat}$.
  \end{enumerate}
\end{proposition}

\begin{proof}
(1) By condition \textbf{(f)} in \cref{def:perfectoidtower} (2), for any $i\geq 0$ the $i$-th Frobenius projection $F_i\colon R_{i+1}/I_0R_{i+1}\to R_i/I_0R_i$ induces an isomorphism $R_{i+1}/I_1R_{i+1}\xr{\cong}R_i/I_0R_i$. Since $R_{i+1}$ and $R_i$ are Zariskian with respect to the $I_1$-adic topology and the $I_0$-adic topology, respectively, we deduce the assertion.

(2) is similar, due to \cref{thm:INStilt} (1).
\end{proof}

\begin{remark}
The ``only if'' part of \cref{prop:towerLocal} (2) was already proved in \cite[Lemma 3.11 (2)]{INS25} by giving the maximal ideal of $R_i^{s.\flat}$ explicitly. The proof of \cref{prop:towerLocal} is based on a simple perspective of pairs.
\end{remark}

%

We say that a perfectoid tower $\textrm{{\boldmath $R$}}=\{R_i\}_{i\geq 0}$ arising from a pair $(R,I_0)$ is \emph{$I_0$-adically complete} (resp.\ \emph{Noetherian}, \emph{local}) if $R_i$ is $I_0$-adically complete (resp.\ Noetherian, local) for any $i\geq 0$.

\begin{proposition}
\label{prop:towerNoeth}
Let $\textrm{{\boldmath $R$}}=\{R_i\}_{i\geq 0}$ be an $I_0$-adically complete perfectoid tower arising from a pair $(R,I_0)$.
  \begin{enumerate}
  \item For any $i\geq 0$, $R_i$ is Noetherian if and only if so is $R_{i+1}$. In particular, if $R_i$ is Noetherian for some $i\geq 0$, then {\boldmath $R$} is Noetherian.
  \item For any $i\geq 0$, $R_i$ is Noetherian if and only if so is $R_i^{s.\flat}$.
  \end{enumerate}
\end{proposition}

The proposition can be verified by an argument similar to that in the proof of \cref{prop:towerLocal}, with the aid of the following lemma, which follows from \cite[Chap.\ III, \S2.10, Corollary 5]{BouAC}.

\begin{lemma}
Let $(A,I)$ and $(B,J)$ be complete pairs such that $I$ and $J$ are finitely generated and there exists a ring isomorphism $A/I\xr{\cong}B/J$. Then $A$ is Noetherian if and only if so is $B$.
\end{lemma}

\begin{remark}
Again, \cref{prop:towerNoeth} (2) was already proved in \cite[Proposition 3.42 (2)]{INS25}. \cref{prop:towerNoeth} just says that the same proof works for adjacent rings $R_i$ and $R_{i+1}$.
\end{remark}

Let us give some examples of perfectoid towers relevant to this paper.

\begin{example}
\label{ex:perfectoidtower}
  \begin{enumerate}
  \item \emph{Perfect towers}. A tower of rings is a perfectoid tower arising from $(R,0)$ if and only if it is a perfect tower. For the proof, see \cite[Lemma 3.24]{INS25}.
  \item \emph{Complete regular local rings}. Let $(R,\m,k)$ be a complete regular local ring of mixed characteristic $(0,p)$. By Cohen's structure theorem, we can write $R=C(k)\llbracket x_1,\ldots,x_d\rrbracket/(p-f)$, where $C(k)$ is the Cohen ring of $k$\footnote{The \emph{Cohen ring} of a field $k$ of characteristic $p>0$ is a unique, up to isomorphism, complete discrete valuation ring $C$ with uniformizer $p$ such that $C/pC=k$.} and $f\in(x_1,\ldots,x_d)$. Fix an algebraic closure of the field of fractions of $R$, and take for each $1\leq j\leq d$ a compatible system of $p$-power roots $\{x_j^{1/p^i}\}_{i\geq 0}$ of $x_j$. Then
  \[
  C(k)\llbracket x_1,\ldots,x_d\rrbracket/(p-f) \hookrightarrow C(k^{1/p})\llbracket x_1^{1/p},\ldots,x_d^{1/p}\rrbracket/(p-f) \hookrightarrow \cdots \hookrightarrow C(k^{1/p^i})\llbracket x_1^{1/p^i},\ldots,x_d^{1/p^i}\rrbracket \hookrightarrow \cdots
  \]
  is a perfectoid tower arising from $(R,p)$, and its tilt is isomorphic to
  \[
  k\llbracket x_1,\ldots,x_d\rrbracket \hookrightarrow k^{1/p}\llbracket x_1^{1/p},\ldots,x_d^{1/p}\rrbracket \hookrightarrow \cdots \hookrightarrow k^{1/p^i}\llbracket x_1^{1/p^i},\ldots,x_d^{1/p^i}\rrbracket \hookrightarrow \cdots.
  \]
  We often write $x_j^{s.\flat}$ instead of $x_j$ to emphasize its correspondence with $x_j$ ($1\leq j\leq d$). For the proof, see \cite[Example 3.62 (1)]{INS25} or \cite[Example 6.1]{Ish26}.
  \item \emph{Complete local log-regular rings}. More generally, let $(R,Q,\alpha)$ be a complete local log-regular ring of mixed characteristic $(0,p)$ with the underlying Noetherian local ring $R$ with residue field $k$. By Kato's structure theorem (\cite[Theorem 2.22]{INS25}), we can write $R=C(k)\llbracket Q\oplus\N^{\oplus r}\rrbracket/(p-f)$, where $r\in\N$ and $f\in C(k)\llbracket Q\oplus \N^{\oplus r}\rrbracket$ has constant term zero. For each $i\geq 0$, let $Q^{(i)}\coloneqq \{\gamma\in Q^{\mathrm{gp}}\mid p^i\gamma\in Q\}$, and define $(\N^{\oplus r})^{(i)}$ similarly. Then
  \[
  C(k)\llbracket Q\oplus\N^{\oplus r}\rrbracket/(p-f)\hookrightarrow C(k^{1/p})\llbracket Q^{(1)}\oplus(\N^{\oplus r})^{(1)}\rrbracket/(p-f) \hookrightarrow \cdots \hookrightarrow C(k^{1/p^i})\llbracket Q^{(i)}\oplus(\N^{\oplus r})^{(i)}\rrbracket/(p-f) \hookrightarrow \cdots
  \]
  is a perfectoid tower arising from $(R,p)$, and its tilt is isomorphic to
  \[
  k\llbracket Q\oplus\N^{\oplus r}\rrbracket \hookrightarrow k^{1/p}\llbracket Q^{(1)}\oplus(\N^{\oplus r})^{(1)}\rrbracket \hookrightarrow \cdots \hookrightarrow k^{1/p^i}\llbracket Q^{(i)}\oplus(\N^{\oplus r})^{(i)}\rrbracket \hookrightarrow \cdots.
  \]
  For the proof, see \cite[Proposition 3.58]{INS25} or \cite[Example 6.2]{Ish26}.
  \item \emph{Modifications in characteristic $p$}. One can construct new (pre)perfectoid towers from old ones by changing their special fibers: if $\textrm{{\boldmath $S$}}=\{S_i,u_i\}_{i\geq 0}$ is a (pre)perfectoid tower arising from a pair $(S,J_0)$, and $R'\to (S/J_0S)_\red$ is any homomorphism of reduced $\F_p$-algebras, then the tower obtained by taking fiber products
\[
R\coloneqq R'\times_{(S/J_0S)_\red} S \to (R')^{1/p}\times_{(S_1/J_0S_1)_\red} S_1 \to \cdots \to (R')^{1/p^i}\times_{(S_i/J_0S_i)_\red} S_i\to \cdots
\]
is a (pre)perfectoid tower arising from $(R,I_0)$, where $I_0=(0)R'\times_{(0)(S/J_0S)_\red}J_0$. For the proof, see \cite[Remark 3.18]{Ha26a}.
  \item \emph{Localizations}. If $\textrm{{\boldmath $R$}}=\{R_i,t_i\}_{i\geq 0}$ is a perfectoid tower arising from a pair $(R,I_0)$ and $S\subset R$ a multiplicative subset with $I_0\cap S=\emptyset$, then $S^{-1}\textrm{{\boldmath $R$}}=\{S^{-1}R_i, S^{-1}t_i\}_{i\geq 0}$ is a perfectoid tower arising from $(S^{-1}R, I_0(S^{-1}R))$. This is a special case of the \'{e}tale base change stability (\cite[Theorem 3.28]{Ha26a} or \cite[Theorem 3.11]{Ha26b}).
  \end{enumerate}

See also \cite{Ish26,IS25} for other constructions and examples of perfectoid towers. As in all these constructions, one can typically take $I_0=(0)$ or $(p)$. Via ramification theory, \cite[\S4]{HIS26} provides constructions where $I_0\neq (0),(p)$.
Later in \cref{s:const}, we will see further constructions where $I_0\neq (0),(p)$.
\end{example}

\begin{remark}
Let $\textrm{{\boldmath $R$}}=\{R_i\}_{i\geq 0}$ be a (pre)perfectoid tower arising from a pair $(R,I_0)$. Then by \cite[Theorem 3.1]{Ha26a} the maximal $I_0$-torsion free quotients $\wt{R_i}\coloneqq R_i/(R_i)_{\Iztor}$ form a (pre)perfectoid tower $\wt{\textrm{{\boldmath $R$}}}=\{\wt{R_i}\}_{i\geq 0}$ arising from $(\wt{R},I_0\wt{R})$ whose tilt consists of the maximal $I_0^{s.\flat}$-torsion free quotients $R_i^{s.\flat}/(R_i^{s.\flat})_{\Izstor}$. Moreover, we have decompositions into fiber products
\[
R_i \xr{\cong} \wt{R_i} \times_{\wt{R_i}/I_0\wt{R_i}} (R_i/I_0R_i),\quad R_i^{s.\flat} \xr{\cong} (\wt{R_i})^{s.\flat}\times_{\wt{R_i}/I_0\wt{R_i}}(R_i/I_0R_i).
\]
Note that the decompositions hold if we replace $R_i/I_0R_i$ and $\wt{R_i}/I_0\wt{R_i}$ by $(R_i/I_0R_i)_{\mathrm{red}}$ and $(\wt{R_i}/I_0\wt{R_i})_{\mathrm{red}}$, respectively. This means that arbitrary perfectiod towers is constructed by modifying $p$-torsion free perfectoid towers by the procedure of \cref{ex:perfectoidtower} (4).
\end{remark}

\section{Main results}
\label{s:regular}

In this section, we prove the results on regular rings and perfectoid towers stated in the introduction.

\subsection{Flatness}

We first study the flatness of morphisms of perfectoid towers.

Let $\textrm{{\boldmath $R$}}=\{R_i,t_i\}_{i\geq 0}$ and $\textrm{{\boldmath $S$}}=\{S_i,u_i\}_{i\geq 0}$ be preperfectoid towers arising from $(R,I_0)$ and $(S,J_0)$, respectively.
Let $\textrm{{\boldmath $\alpha$}}=\{\alpha_i\}_{i\geq 0}\colon \textrm{{\boldmath $R$}} \to \textrm{{\boldmath $S$}}$ be a morphism of towers of rings such that $\alpha_0(I_0)\subset J_0$. Then there exists a unique morphism of towers of rings $\textrm{{\boldmath $\alpha$}}^\flat=\{\alpha_i^{s.\flat}\}_{i\geq 0}\colon \textrm{{\boldmath $R$}}^\flat \to \textrm{{\boldmath $S$}}^\flat$ such that the diagram
\[
\begin{tikzcd}[column sep=large]
R_i^{s.\flat} \rar["\alpha_i^{s.\flat}"] \dar["\pr_m"'] & S_i^{s.\flat} \dar["\pr_m"] \\
R_{i+m}/I_0R_{i+m} \rar["\ol{\alpha_{i+m}}"] & S_{i+m}/J_0S_{i+m}
\end{tikzcd}
\]
commutes for any $i,m\geq 0$ (\cite[Definition 2.12]{Ha26a}). We call $\textrm{{\boldmath $\alpha$}}^\flat$ the \emph{tilt of {\boldmath $\alpha$}} (\emph{associated to $(R,I_0)$ and $(S,J_0)$}).

Here we include the following fact on flatness quoted from \cite{GR}.

\begin{lemma}[{\cite[Corollary 3.4.22]{GR}, \cite[\href{https://stacks.math.columbia.edu/tag/08KQ}{Tag 08KQ}]{stacks-project}}]
\label{lem:flatglue}
Consider a commutative diagram of rings
\[
\begin{tikzcd}[sep=small]
 & A' \ar[rr] \ar[dd] & & A'_2 \ar[dd] \\
A \ar[rr,crossing over] \ar[dd] \ar[ru] & & A_2 \ar[dd] \ar[ru] & \\
 & A'_1 \ar[rr] & & B' \\
A_1 \ar[rr,"f_1"'] \ar[ru] & & B \ar[ru] \ar[from=2-3,crossing over]
\end{tikzcd}
\]
where $f_1$ is surjective, the front and back faces are cartesian, and the bottom and right faces are cocartesian. Then $A\to A'$ is flat if and only if $A_1\to A'_1$ and $A_2\to A'_2$ are flat.
\end{lemma}

We will apply \cref{lem:flatglue} to the following situation.

\begin{lemma}
\label{lem:flatglue2}
Let $(R,I)$ be a pair, and $S$ an $R$-algebra. Let $\wt{R}\coloneqq R/R_{\Itor}$ and $\wt{S}\coloneqq S/S_{\Itor}$. Consider the commutative diagram of rings
\[
\begin{tikzcd}[column sep=tiny, row sep=small]
 & S \ar[rr] \ar[dd] & & \wt{S} \ar[dd] \\
R \ar[rr,crossing over] \ar[dd] \ar[ru] & & \wt{R} \ar[dd] \ar[ru] & \\
 & S/IS \ar[rr] & & \wt{S}/I\wt{S} \\
R/I \ar[rr] \ar[ru] & & \wt{R}/I\wt{R} \ar[ru] \ar[from=2-3,crossing over]
\end{tikzcd}
\]
  \begin{enumerate}
  \item The right face is always cocartesian.
  \item If $R_{\Itor}\cap I=(0)$ \emph{(}resp.\ $S_{\Itor}\cap IS=(0)$\emph{)}, then the front \emph{(}resp.\ back\emph{)} face is cartesian.
  \item If $R\to S$ is flat, then the bottom face is cocartesian.
  \end{enumerate}
\end{lemma}

\begin{proof}
(1) Obvious.

(2) follows from \cite[Lemma 3.5]{Ha26a}.

(3) If $R\to S$ is flat, then the canonical map $\wt{R}\otimes_RS \to \wt{S}$ is an isomorphism, and thus we obtain the desired isomorphism $\wt{R}/I\wt{R} \otimes_{R/I} (S/IS) \cong \wt{R}\otimes_RS\otimes_R(R/I) \xr{\cong} \wt{S}\otimes_R(R/I) \xr{\cong} \wt{S}/I\wt{S}$.
\end{proof}

Now we have the promised result.

\begin{proposition}
\label{prop:tilting flat}
Let $\textrm{{\boldmath $R$}}=\{R_i\}_{i\geq 0}\to \textrm{{\boldmath $S$}}=\{S_i\}_{i\geq 0}$ be a morphism of towers of rings, where {\boldmath $R$} and {\boldmath $S$} are preperfectoid towers arising from pairs $(R,I_0)$ and $(S,J_0)$. Suppose that $J_0=I_0S$ and $J_1=I_1S$. Fix $i\geq 0$. Suppose that either one of the following conditions is satisfied.
  \begin{enumerate}
  \item Both $R_i$ and $S_i$ are Noetherian, and $S_i$ is $I_0$-adically Zariskian.
  \item $R_i$ is $I_0$-adically complete Noetherian, and $S_i$ is $I_0$-adically complete.
  \end{enumerate}
Then $R_i\to S_i$ is flat if and only if $R_i^{s.\flat}\to S_i^{s.\flat}$ is flat.
\end{proposition}

\begin{proof}
By shifting, we may assume $i=0$. Let $\wt{R}\coloneqq R/R_{\Iztor}$ and $\wt{S}\coloneqq S/S_{\Iztor}$. Since $J_0^{s.\flat}\coloneqq\Ker(S^{s.\flat}\xr{\pr_0}S/J_0)=I_0^{s.\flat}S^{s.\flat}$ (\cite[Lemma 2.14]{Ha26a}), we have the following commutative diagrams with the same bottom faces.
\[
\begin{tikzcd}[column sep=tiny, row sep=small]
 & S \ar[rr] \ar[dd] & & \wt{S} \ar[dd] \\
R \ar[rr,crossing over] \ar[dd] \ar[ru] & & \wt{R} \ar[dd] \ar[ru] & \\
 & S/I_0S \ar[rr] & & \wt{S}/I_0\wt{S}, \\
R/I_0 \ar[rr] \ar[ru] & & \wt{R}/I_0\wt{R} \ar[ru] \ar[from=2-3,crossing over]
\end{tikzcd}
\qquad
\begin{tikzcd}[column sep=tiny, row sep=small]
 & S^{s.\flat} \ar[rr] \ar[dd] & & (\wt{S})^{s.\flat} \ar[dd] \\
R^{s.\flat} \ar[rr,crossing over] \ar[dd] \ar[ru] & & (\wt{R})^{s.\flat} \ar[dd] \ar[ru] & \\
 & S/I_0S \ar[rr] & & \wt{S}/I_0\wt{S} \\
R/I_0 \ar[rr] \ar[ru] & & \wt{R}/I_0\wt{R} \ar[ru] \ar[from=2-3,crossing over]
\end{tikzcd}
\]
By \cref{lem:flatglue2} (1) (2), the right faces are cocartesian and the front and back faces are cartesian. Moreover, if either one of $R\to S$ or $R^{s.\flat}\to S^{s.\flat}$ is flat, then the bottom face is cocartesian by \cref{lem:flatglue2} (3). Hence we can apply \cref{lem:flatglue} to both cubes above, and thus we may replace $R\to S$ by $\wt{R}\to \wt{S}$. (The condition (2) is preserved by \cite[Proposition 3.10]{Ha26a}).

(1) By a consequence of the local criterion of flatness (\cite[Exercise 22.3]{Mat2}), the flatness of either one of $R\to S$ and $R^{s.\flat}\to S^{s.\flat}$ is equivalent to that of $R/I_0\to S/I_0S$. This completes the proof in this case.

(2) is similar, due to Bhatt's result \cite[Proposition 5.1]{Bha18}.
\end{proof}

\begin{remark}
Using the local criterion of flatness (\cite[Theorem 22.3]{Mat2}), \cref{prop:tilting flat} in the case (1) also follows from the isomorphism of conormal cones $\gr_{I_0^{s.\flat}}(R^{s.\flat})\xr{\cong} \gr_{I_0}(R)$ (\cite[Proposition 3.4]{Ha26b}).
\end{remark}

As a corollary, we have the following result.

\begin{corollary}
\label{cor:flatrigid}
Let $\textrm{{\boldmath $R$}}=\{R_i,t_i\}_{i\geq 0}$ be a Noetherian perfectoid tower arising from a pair $(R,I_0)$.
  \begin{enumerate}
  \item For any $i\geq 0$, $t_i$ is flat if and only if $t_i^{s.\flat}$ is flat.
  \item If $t_i$ is flat for some $i\geq 0$, then $t_i$ is flat for any $i\geq 0$.
  \end{enumerate}
\end{corollary}

\begin{proof}
(1) Apply \cref{prop:tilting flat} (2) to the morphism of towers of rings $\{t_{i+j}\}_{j\geq 0} \colon \{R_{i+j}\}_{j\geq 0} \to \{R_{i+j+1}\}_{j\geq 0}$.

(2) By (1), we may replace {\boldmath $R$} by $\textrm{{\boldmath $R$}}^\flat$. Then all $t_i$ are isomorphic to the absolute Frobenius of $R$, and thus the assertion is clear.
\end{proof}

We turn to the regularity in terms of perfectoid towers. Recall that Kunz's theorem asserts that a Noetherian $\F_p$-algebra is regular if and only if its absolute Frobenius is flat. In his paper \cite{Lur23}, Lurie proved the following mixed characteristic variant of Kunz's theorem, which was suggested by Gabber:

\begin{theorem}[{Gabber--Lurie \cite[Theorem 6]{Lur23}}]
\label{thm:LurKunz}
Let $R$ be a Noetherian ring, and $\pi\in R$ an element such that $p\in\pi^pR$. Consider the following conditions.
  \begin{enumerate}
  \item For every maximal ideal $\m\subset R$ containing $\pi$, the local ring $R_\m$ is regular.
  \item The ring map $R/\pi R\to R/\pi^pR$ induced by the absolute Frobenius of $R/\pi^pR$ is flat.
  \end{enumerate}
Then the implication ``\emph{(1)} $\Rightarrow$ \emph{(2)}'' holds. If $R$ is $\pi$-torsion free, then we also have ``\emph{(2)} $\Rightarrow$ \emph{(1)}.''
\end{theorem}

\begin{remark}
  \begin{enumerate}
  \item Note that the proof of ``(1) $\Rightarrow$ (2)'' in \cite[Theorem 6]{Lur23} does not use the assumption that $R$ is $\pi$-torsion free. This implication also follows from prismatic Kunz's theorem (\cite[Remark 5.16]{IN26}).
  \item As explained in \cite[Warning 2]{Lur23}, the implication ``(2) $\Rightarrow$ (1)'' is not true without the assumption that $R$ is $\pi$-torsion free.
  \end{enumerate}
\end{remark}

Now we apply \cref{thm:LurKunz} to characterize regularity by the flatness of transition maps.

\begin{theorem}
\label{thm:regularFlat}
Let $\textrm{{\boldmath $R$}}=\{R_i,t_i\}_{i\geq 0}$ be a perfectoid tower arising from a pair $(R,I_0)$. Fix $i\geq 0$, and assume that $R_i$ and $R_{i+1}$ are Noetherian. Then the following conditions are equivalent.
  \begin{enumerate}
  \item The Noetherian ring $R_i$ is regular.
  \item The transition map $t_i\colon R_i\to R_{i+1}$ is flat.
  \end{enumerate}
\end{theorem}

\begin{proof}
By shifting, we may assume $i=0$. As $R$ is $I_0$-adically Zarisikian, we may replace $R$ by the localization $R_\m$ at a maximal ideal $\m$ (\cref{ex:perfectoidtower} (4)). Then $R_1$ is also local by \cref{prop:towerLocal} (1).

(1) $\Rightarrow$ (2): By \cite[Theorem 23.1]{Mat2}, we only have to show that $R_1$ is Cohen--Macaulay and the equality $\dim (R_1)=\dim(R) + \dim(R_1/\m R_1)$ holds. In view of \cref{lem:localtower} (3), it is enough to prove the following claims.
  \begin{itemize}
  \item $R$ is Cohen--Macaulay if and only if $R_1$ is Cohen--Macaulay.
  \item The equality $\dim(R)=\dim (R_1)$ holds.
  \end{itemize}
Since the regular local ring $R$ is $I_0$-torsion free, so is $R_1$ by \textbf{(g)}. Then both claims are immediate from the isomorphism $R_1/I_1\xr{\cong}R/I_0$, although these claims are true even if $R$ has $I_0$-torsion (\cite[Theorem 4.13 (2)]{Ha26a}).

(2) $\Rightarrow$ (1): Since $t_0$ is local by \cref{lem:localtower} (1), it is faithfully flat. Hence, by faithfully flat descent, it suffices to show that $R_1$ is regular. By (the proof of) \cref{cor:flatrigid} (2), $t_0^{s.\flat}\colon R^{s.\flat}\to R_1^{s.\flat}$ is flat. This means that the local ring $R^{s.\flat}$ is regular by Kunz's theorem. In particular, $R^{s.\flat}$ is an integral domain. Then $R_1$ is $I_0$-torsion free (\cite[Proposition-Definition 3.15]{HIS26}). Hence, due to \cref{thm:LurKunz}, it is enough to prove the following: the map $R_1/I_1 \to R_1/I_0R_1$ induced by the absolute Frobenius of $R_1/I_0R_1$ is flat. But the stated map is the composite
\[
R_1/I_1 \xr{\cong} R/I_0 \xr{\ol{t_0}} R_1/I_0R_1,
\]
where the isomorphism is induced by the $0$-th Frobenius projection $F_0\colon R_1/I_0R_1 \to R/I_0$ (cf.\ \textbf{(f-2)}). Hence the desired flatness is equivalent to that of $\ol{t_0}$. Since $t_0$ is assumed to be flat, so is the base change $\ol{t_0}$.
\end{proof}

\begin{remark}
According to Kunz's theorem, one would expect that $R$ is regular if and only if the $0$-th Frobenius projection $F_0\colon R_1/I_0R_1\to R/I_0$ is flat. However, the flatness of $F_0$ neither implies nor is implied by the regularity of $R$:
  \begin{enumerate}
  \item Consider the perfect tower $R\xr{\varphi}R\xr{\varphi}R\xr{\varphi}\cdots$ for a reduced $\F_p$-algebra $R$ (\cref{ex:perfectoidtower} (1)). Then the Frobenius projection is the identity map $\id_R$, which is flat. But $R$ may be singular.
  \item Consider the perfectoid tower $\Z_p\hookrightarrow \Z_p[p^{1/p}] \hookrightarrow \Z_p[p^{1/p^2}]\hookrightarrow\cdots$ associated to the regular local ring $\Z_p$ (\cref{ex:perfectoidtower} (2)). If the $i$-th Frobenius projection $\Z_p[p^{1/p^{i+1}}]/(p) \to \Z_p[p^{1/p^i}]/(p)$ is flat, then the absolute Frobenius of $\Z_p[p^{1/p^{i+1}}]/(p)$ would be flat. It means that $\Z_p[p^{1/p^{i+1}}]/(p)$ is regular by Kunz's theorem. But $\Z_p[p^{1/p^{i+1}}]/(p)$ is not even reduced.
  \end{enumerate}
\end{remark}

Now let us state an immediate but important corollary of \cref{thm:regularFlat}.

\begin{corollary}
Let $\textrm{{\boldmath $R$}}=\{R_i\}_{i\geq 0}$ be a Noetherian perfectoid tower of arising from a pair $(R,I_0)$.
  \begin{enumerate}
  \item For any $i\geq 0$, then $R_i$ is regular if and only if $R_i^{s.\flat}$ is regular.
  \item If $R_i$ is regular for some $i\geq 0$, then $R_i$ is regular for any $i\geq 0$.
  \end{enumerate}
\end{corollary}

\begin{proof}
This is a direct consequence of \cref{thm:regularFlat,cor:flatrigid}.
\end{proof}

Next we show the following Kunz-type criterion of regular rings in terms of perfectoid towers.

\begin{theorem}
\label{thm:towerKunz}
Let $R$ be a Noetherian local ring of residue characteristic $p$. Then the following conditions are equivalent.
  \begin{enumerate}
  \item $R$ is regular.
  \item There exists a Noetherian perfectoid tower $\textrm{{\boldmath $R$}}=\{R_i,t_i\}_{i\geq 0}$ arising from a pair $(R,I_0)$ such that for any (or, equivalently, some) $i\geq 0$ the transition map $t_i\colon R_i\to R_{i+1}$ is flat.
  \end{enumerate}
\end{theorem}

\begin{proof}
(1) $\Rightarrow$ (2): Once a regular system of parameters of $R$ is chosen, we only have to construct a perfectoid tower as in \cref{ex:perfectoidtower} (2).

(2) $\Rightarrow$ (1): Since the Noetherian ring $R_1$ is local by \cref{prop:towerLocal}, we can apply \cref{thm:regularFlat} to deduce that $R$ is regular.
\end{proof}

\begin{remark}
\label{rem:anotherpf}
\cref{thm:towerKunz}---except for the parenthetical assertion---can be deduced from the $p$-adic Kunz's theorem of Bhatt--Iyenga--Ma (\cite[Theorem 4.7]{BIM19}). Indeed, the proof there shows the implication``(1) $\Rightarrow$ (2).'' Conversely, if there exists a perfectoid tower $\textrm{{\boldmath $R$}}=\{R_i,t_i\}_{i\geq 0}$ as in (2), then the perfectoid $R$-algebra $R\to\wh{R_\infty}$ is faithfully flat by \cite[Lemma A.3]{HIS26}.
\end{remark}

By using \cref{thm:towerKunz}, we give another proof of the fact that for any regular local ring $R$ of residue characteristic $p$ and any prime ideal $\p$ containing $p$ the localization $R_\p$ is also a regular local ring. This is a generalization of \cite[Corollary 2.2]{Kun76}, which proves the same statement for regular local rings over $\F_p$ by using Kunz's theorem. Note
that we do not use Auslander--Buchsbaum--Serre's regularity criterion because so does \cref{thm:towerKunz}.

\begin{proposition}
\label{prop:RLRloc}
Let $R$ be a regular local ring of residue characteristic $p$, and $\p\subset R$ a prime ideal containing $p$. Then the local ring $R_\p$ is also regular.
\end{proposition}

\begin{proof}
By a similar argument as in the proof of \cite[Proposition 5.17]{IN26}, we may assume that the local ring $R$ is complete. Then we know that there exists a perfectoid tower $\{R_i,t_i\}_{i\geq 0}$ arising from $(R,p)$ (\cref{ex:perfectoidtower} (2)). Since $\{R_{i,\p},t_{i,\p}\}_{i\geq 0}$ is a perfectoid tower arising from $(R_\p,p)$ by \cref{ex:perfectoidtower} (5),  Then we apply \cref{thm:towerKunz} to conclude that the local ring $R_\p$ is regular.
\end{proof}

Note that R.~Ishizuka and K.~Nakazato also gave another proof of \cref{prop:RLRloc} without Auslander--Buchsbaum--Serre's regularity criterion (\cite[Proposition 5.17]{IN26}). While their proof relies on the methods from homotopy theory, our proof above is based on an elementary argument, such as that appeared in the proof of \cref{thm:LurKunz}.

\subsection{Weakly proregular sequences}

In this section, we study weakly proregular sequences in perfectoid towers.

\begin{definition}[{\cite[Definition 2.3]{Sch03}}]
Let $A$ be a ring, and $M$ an $A$-module.
We say that a sequence $\bm{x}=x_1,\ldots,x_r$ of elements in $A$ is \emph{weakly proregular on $M$} if for each $q>0$ the projective system of Koszul homology $\{H_q(\bm{x}^n;M)\}_{n\geq 0}$ is pro-zero, where $\bm{x}^n\coloneqq x_1^n,\ldots,x_r^n$. In other words, for each $n\geq 0$, there exists an integer $m\geq n$ such that the canonical map $H_q(\bm{x}^m;M)\to H_q(\bm{x}^n;M)$ is zero.
\end{definition}

For example, every sequence of elements in a Noetherian ring $A$ is weakly proregular on $A$.
In general, the notion of weakly proregular sequences is closely related to the relationship between local cohomology and \v{C}ech cohomology. The vanishing of local cohomology implies the following rigidity result on $\Tor$-modules:

\begin{proposition}[{\cite[Proposition 3.3]{CIM19}; see also \cite[2.3]{BIM19}}]
\label{prop:CIM}
Let $(R,\m,k)$ be a Noetherian local ring, and $M$ a \emph{(}not necessarily finitely generated\emph{)} an $R$-module. Set
\[
s(M)\coloneqq \sup\{q\mid H_\m^q(M^\vee)\neq 0\}
\]
where $M^\vee=\Hom_R(M,E_R(k))$ is the Matlis dual of $M$. If $\Tor^R_q(M,k)=0$ for some $q\geq s(M)$, then $\Tor^R_{q'}(M,k)=0$ for all $q'\geq q$.
\end{proposition}

Bhatt--Iyengar--Ma combined \cref{prop:CIM} with their precise criterion for detecting finiteness of flat dimension to obtain a sharper characterization in terms of perfectoid algebras. We summarize an important consequence as follows:

\begin{lemma}
\label{lem:BIM}
Let $(R,\m,k)$ be a Noetherian local ring, and $A$ a perfectoid $R$-algebra. Suppose that a system of parameters of $R$ is weakly proregular on $A$.
Then $R$ is regular if $\Tor^R_q(A,k)=0$ for some $q\geq 1$.
\end{lemma}

\begin{proof}
Since the Matlis dual $A^\vee=\Hom_R(A,E_R(k))$ is an injective $A$-module, and it follows from the assumption that $H_\m^q(A)=0$ for any $q\geq 1$ (\cite[Lemma 4.10]{BIM19}). Then the vanishing $\Tor^R_q(A,k)=0$ for some $q\geq 1$ implies that $\Tor^R_q(A,k)=0$ for all $q\geq 1$ by \cref{prop:CIM}. Then $\Tor^R_q(k,k)=0$ because $A$ is perfectoid (\cite[Theorem 4.1]{BIM19}).
\end{proof}

In \cite{BIM19}, \cref{lem:BIM} applied in two cases: $A=R^+$ is the absolute integral closure of a domain $R$,\footnote{More precisely, we should take $A$ to be the $p$-adic completion of $R^+$. However, \cref{lem:BIM} holds in the case $A=R^+$ (see \cite[Remark 4.3]{BIM19}).} and $A=R_{\mathrm{perf}}$ is the perfect closure of an $\F_p$-algebra $R$.
We will consider more general perfectoid rings.

Let us recall the following basic notion for perfectoid rings.
For a $p$-adically complete ring $A$, we define $A^\flat\coloneqq\varprojlim_\varphi(A/pA)$, which is called the \emph{tilt} of $A$ when $A\not\supset\F_p$. Due to the $p$-adic completeness of $A$, the canonical projection $A\to A/pA$ induces an isomorphism of multiplicative monoids $\varprojlim_{x\mapsto x^p}A \xr{\cong} A^\flat$, and let
\[
\sharp\colon A^\flat \to A;\quad x\mapsto x^\sharp
\]
denote the resulting multiplicative projection onto the $0$-th coordinate. Thus elements of the form $x^\sharp$ ($x\in A^\flat$) precisely are elements $y$ having a compatible system of $p$-power roots $\{y^{1/p^n}\}_{n\geq 0}$. Here we include the following result of Gabber--Ramero.

\begin{proposition}[{\cite[Propositions 7.8.25 (i) and 16.4.10 (ii)]{GR24}}]
\label{prop:WPR_GR}
Let $A$ be a perfectoid ring, and $\bm{x}=x_1,\ldots,x_r$ a sequence of elements in $A^\flat$. Then $\bm{x}^\sharp\coloneqq x_1^\sharp,\ldots,x_r^\sharp$ is a weakly proregular sequence on $A$.
\end{proposition}

Once given a perfectoid tower, we have the following result.

\begin{proposition}
\label{prop:WPRsop}
Let $\textrm{{\boldmath $R$}}=\{R_i\}_{i\geq 0}$ be a perfectoid tower arising from a pair $(R,I_0)$, where $R$ is a Noetherian local ring. Assume that either $R$ is $I_0$-torsion free or $I_0=(0)$.
Let $\bm{x}=x_1,\ldots,x_d$ be a system of parameters of $R$. Then $\bm{x}$ is weakly proregular on any perfectoid $\wh{R_\infty}$-algebra $A$.
\end{proposition}

\begin{proof}
By \cite[Corollary 3.3]{Sch03}, it suffices to verify that there is some choice of an s.o.p.\ that is weakly proregular on $A$.
By assumption, we can take an s.o.p.\ $\bm{x}=x_1,\ldots,x_d$ of $R^{s.\flat}$ such that $I_0^{s.\flat}=(x_1)$. Now consider the monoidal map $\sharp\colon \left(\wh{R_\infty}\right)^\flat\to \wh{R_\infty}$, which maps $R^{s.\flat}$ into $R+I_0\wh{R_\infty}$ (see \cite{HIS26}). Hence we can write
\[
x_k^\sharp=y_k + f_0z_k\quad (y_k\in R,\ z_k\in\wh{R_\infty}),
\]
where $f_0\in R$ is a generator of $I_0$. Since $\sharp\colon\left(\wh{R_\infty}\right)^\flat\to\wh{R_\infty}$ induces an isomorphism $R^{s.\flat}/I_0^{s.\flat}\xr{\cong}R/I_0$ (\cite[Proposition 3.28]{HIS26}), we have $R^{s.\flat}/(\bm{x}) \xr{\cong} R/(\bm{y})$, and thus $\bm{y}$ is an s.o.p.\ of $R$. Moreover, we have $\bm{x}^\sharp A=\bm{y}A$, and since $\bm{x}^\sharp$ is weakly proregular in $A$ by \cref{prop:WPR_GR}, so is $\bm{y}$. This completes the proof.
\end{proof}

The following theorem is a direct consequence of \cref{prop:WPRsop} and \cref{lem:BIM}.

\begin{theorem}
\label{thm:WPRreg}
Let $(R,\m,k)$ be a Noetherian local ring of residue characteristic $p$. Suppose that there exists a perfectoid tower $\textrm{{\boldmath $R$}}=\{R_i\}_{i\geq 0}$ arising from a pair $(R,I_0)$ such that either $R$ is $I_0$-torsion free or $I_0=(0)$. If there exists a perfectoid $\wh{R_\infty}$-algebra $A$ such that $\Tor^R_q(A,k)=0$ for some $q\geq 1$, then $R$ is regular.
\end{theorem}

In positive characteristic, we always have perfectoid towers, namely, perfect towers (\cref{ex:perfectoidtower} (1)). Hence we obtain the following corollary, which mildly generalizes \cite[Theorem 4.13 (1)]{BIM19}.

\begin{corollary}
Let $(R,\m,k)$ be a reduced Noetherian local $\F_p$-algebra. If there exists a perfect $R$-algebra $A$ such that $\Tor^R_q(A,k)=0$ for some $q\geq 1$, then $R$ is regular.
\end{corollary}

\section{Further constructions of perfectoid towers from regular rings}
\label{s:const}

In this section, we construct perfectoid towers arising from the pair consisting of regular rings and ideals not generated by $p$. This construction is based on the perfectoid rings given in \cite[Example 2.1.4]{Nak20}, which do not fit into Scholze's original definition of perfectoid algebras.

For a pair $(A,I)$, the associated \emph{Rees algebra} is the graded ring
\[
R(A,I) = \bigoplus_{n\geq 0}I^n
\]
(where $I^0=A$). Suppose $x_1,\ldots,x_r$ generate $I$. Then $\Proj R(A,I)$, which is the blow-up of the affine scheme $\Spec A$ along $I$, is covered by the affine open subsets $\{\Spec R(A,I)_{(x_j)}\}_{j=1,\ldots,r}$, where $R(A,I)_{(x_j)}$ is the graded localization by the multiplicative subset $\{1,x_j,x_j^2,\ldots\}$. Note that if $x_1,\ldots,x_r$ is a regular sequence, then the Rees algebra is given by
\[
R(A,I) = A[Y_1,\ldots,Y_r]/I_2
\begin{pmatrix}
x_1 & \cdots & x_r \\
Y_1 & \cdots & Y_r
\end{pmatrix},
\]
where $Y_1,\ldots,Y_r$ is a set of variables over $A$ (\cite[Corollary of Theorem 2, Theorem 4]{Bar73}). 

\begin{proposition}
\label{prop:Rees}
Let $(A,\m)$ be an unramified complete regular local ring of mixed characteristic $(0,p)$. Take a regular system of parameters $p,x_2,\ldots,x_d$ of $A$, and consider the associated perfectoid tower $\{A_i\}_{i\geq 0}$ \emph{(}\cref{ex:perfectoidtower} (2)\emph{)}. For each $i\geq 0$, let $\m_i$ denote the maximal ideal of $A_i$.
  \begin{enumerate}
  \item $\textrm{{\boldmath $R$}}\coloneqq \{R(A_i,\m_i)\}_{i\geq 0}$ is a preperfectoid tower arising from $(R(A,\m),p)$.
  \item The $i$-th perfectoid pillar $I_i\subset R(A_i,\m_i)$ is $(p^{1/p^i})$, the ideal generated by the degree-$1$ element $p^{1/p^i}\in R(A_i,\m_i)$.
  \item The tilt $\textrm{{\boldmath $R$}}^\flat$ of {\boldmath $R$} is isomorphic to the tower $\{R(A_i^{s.\flat},\m_i^{s.\flat})\}_{i\geq 0}$, where $\m_i^{s.\flat}$ is the maximal ideal of $A_i^{s.\flat}=(A/\m)\llbracket p^{s.\flat},x_2^{s.\flat},\ldots,x_d^{s.\flat}\rrbracket$ for all $i\geq 0$.
  \item The small tilt $I_i^{s.\flat}\subset R(A_i^{s.\flat},\m_i^{s.\flat})$ is $((p^{s.\flat})^{1/p^i})$, the ideal generated by the degree-$1$ element $(p^{s.\flat})^{1/p^i}\in R(A_i^{s.\flat},\m_i^{s.\flat})$.
  \end{enumerate}
\end{proposition}

\begin{proof}
By Cohen's structure theorem, we may assume that $A=C(k)\llbracket x_2,\ldots,x_d\rrbracket$, where $C(k)$ is the Cohen ring of $k=A/\m$. Then $A_i=C(k^{1/p^i})\llbracket p^{1/p^i},x_2^{1/p^i},\ldots,x_d^{1/p^i}\rrbracket$. Since $p^{1/p^i}, x_2^{1/p^i},\ldots,x_d^{1/p^i}$ is a regular sequence in $A_i$, the $(\m_i\oplus\m_i\oplus\m_i^2\oplus\cdots)$-adically completed Rees algebra is given by
\begin{align*}
R(A_i,\m_i)^\wedge &= A_i\llbracket Y_1^{1/p^i},\ldots,Y_d^{1/p^i}\rrbracket/I_2
\begin{pmatrix}
p^{1/p^i} & x_2^{1/p^i} & \cdots & x_d^{1/p^i} \\
Y_1^{1/p^i} & Y_2^{1/p^i} & \cdots & Y_d^{1/p^i}
\end{pmatrix}
 \\
&\cong 
\frac{C(k^{1/p^i})\llbracket X_1^{1/p^i},\ldots,X_d^{1/p^i},Y_1^{1/p^i},\ldots,Y_d^{1/p^i}\rrbracket}
{I_2
\begin{pmatrix}
X_1^{1/p^i} & X_2^{1/p^i} & \cdots & X_d^{1/p^i} \\
Y_1^{1/p^i} & Y_2^{1/p^i} & \cdots & Y_d^{1/p^i}
\end{pmatrix}
+(p-X_1)}.
\end{align*}
Let $e_1,\ldots,e_{2d}$ be the standard basis of $\N^{\oplus 2d}$, and set
\[
\xi_j\coloneqq e_1 + e_{j+1} + \cdots + \check{e}_{d+1} + \cdots + e_{d+j},\quad \eta_j\coloneqq  e_{j+1} + \cdots + e_{d+j}\quad (1\leq j\leq d).
\]
Consider the submonoid $Q\coloneqq \langle \xi_1,\ldots,\xi_d,\eta_1,\ldots,\eta_d\rangle\subset\N^{\oplus 2d}$. Then $R(A_i,\m_i)^\wedge \cong C(k^{1/p^i})\llbracket Q^{(i)}\rrbracket/(p-\xi_1)$, and the tower $\{R(A_i,\m_i)^\wedge\}_{i\geq 0}$ is isomorphic to the perfectoid tower associated to the complete local log-regular ring $C(k)\llbracket Q\rrbracket/(p-\xi_1)$ (\cref{ex:perfectoidtower} (3)). From this, all assertions follow easily.
\end{proof}

\begin{corollary}
\label{cor:Rees2}
Consider the situation as in \cref{prop:Rees}, and fix $2\leq k\leq d$.
  \begin{enumerate}
  \item $\textrm{{\boldmath $R_0$}}\coloneqq \{R(A_i,\m_i)_{(x_k)}\}_{i\geq 0}$ is a preperfectoid tower arising from $(R(A,\m)_{(x_k)},\tfrac{p}{x_k})$.
  \item The $i$-th perfectoid pillar $J_i\subset R(A_i,\m_i)_{(x_k)}$ is $\left(\tfrac{p^{1/p^i}}{x_k^{1/p^i}}\right)$.
  \item The tilt of {\boldmath $R_0$} is isomorphic to $\{R(A_i^{s.\flat},\m_i^{s.\flat})_{(x_k^{s.\flat})}\}_{i\geq 0}$
  \item The small tilt $J_i^{s.\flat}\subset R(A_i^{s.\flat},\m_i^{s.\flat})_{(x_k^{s.\flat})}$ is $\left(\tfrac{(p^{s.\flat})^{1/p^i}}{(x_k^{s.\flat})^{1/p^i}}\right)$.
  \end{enumerate}
\end{corollary}

\begin{proof}
By \cref{prop:Rees} and \cref{ex:perfectoidtower} (5), $\{G_i\coloneqq R(A_i,\m_i)_{x_k}\}_{i\geq 0}$ is a preperfectoid tower arising from $(R(A,\m)_{x_k},p)$. Since
\[
G_{i+j}/I_iG_{i+j} = \bigoplus_{d\in\Z}[G_{i+j}]_d/(I_iG_{i+j}\cap[G_{i+j}]_d)
\]
and
\[
I_iG_{i+j}\cap[G_{i+j}]_0 = \left( \frac{p^{1/p^i}}{x_k^{1/p^i}} \right) = (I_i\cap[G_i]_0)\cdot[G_{i+j}]_0,
\]
for any $i,j\geq 0$, all assertions follow easily from \cref{prop:Rees}.
\end{proof}

\begin{example}
\label{ex:ZpT}
Consider the unramified complete regular local ring $A=\Z_p\llbracket T\rrbracket$. Then
\[
A[\tfrac{p}{T}] \hookrightarrow A_1[\left(\tfrac{p}{T}\right)^{1/p}] \hookrightarrow \cdots \hookrightarrow A_i[\left(\tfrac{p}{T}\right)^{1/p^i}] \hookrightarrow \cdots
\]
is a preperfectoid tower arising from $(\Z_p\llbracket T\rrbracket[\tfrac{p}{T}], \tfrac{p}{T})$, where $A_i\coloneqq A[p^{1/p^i},T^{1/p^i}]$.
\end{example}

This example also arises from the following construction (see \cref{ex:ZpTU} (1)).

\begin{construction}
\label{const:monomial}
Let $R=W(k)\llbracket x_1,\ldots,x_d\rrbracket/(p-f)$ be a ramified complete regular local ring of mixed characteristic $(0,p)$, where $W(k)$ is the Witt ring of a perfect field $k$ of characteristic $p$ and $f\in(x_1,\ldots,x_d)^2\setminus\{0\}$. Suppose that $f$ is a monomial. Then we have the compatible system of $p$-power roots $\{f^{1/p^i} \in W(k)\llbracket x_1,\ldots,x_d\rrbracket[x_1^{1/p^i},\ldots,x_d^{1/p^i}]\}_{i\geq 0}$.
We define a tower of rings $\textrm{{\boldmath $R$}}=\{R_i,t_i\}_{i\geq 0}$ as follows.
  \begin{itemize}
  \item For any $i\geq 0$, let
  \[
  R_i\coloneqq R[x_1^{1/p^i},\ldots,x_d^{1/p^i}]/(p^{1/p^i}-f^{1/p^i}).
  \]
  \item For any $i\geq 0$, the transition map $t_i\colon R_i\to R_{i+1}$ is given by the canonical inclusion
  \[
  R[x_1^{1/p^i},\ldots,x_d^{1/p^i}] \hookrightarrow R[x_1^{1/p^{i+1}},\ldots,x_d^{1/p^{i+1}}].
  \]
  Here the inclusion $(p^{1/p^i}-f^{1/p^i}) \subset (p^{1/p^{i+1}}-f^{1/p^{i+1}})$ can be checked by considering separately the cases $p=2$ or $p\geq 3$.
  \end{itemize}
\end{construction}

\begin{proposition}
\label{prop:g}
Consider the situation as in \cref{const:monomial}.
Assume that $k\llbracket x_1,\ldots,x_d\rrbracket/(\ol{f})$ is reduced, where $\ol{f}$ is the image of $f$ in $k\llbracket x_1,\ldots,x_d\rrbracket$.
Let $g\in W(k)\llbracket x_1,\ldots,x_d\rrbracket$ with $f\in(g)$.
  \begin{enumerate}
  \item {\boldmath $R$} is a preperfectoid tower arising from $(R,g)$.
  \item The $i$-th perfectoid pillar $I_i\subset R_i$ of {\boldmath $R$} is $g^{1/p^i}R_i$, where $g^{1/p^i}$ is the $p^i$-th root of $g$.
  \item The tilt $\textrm{{\boldmath $R$}}^\flat$ of {\boldmath $R$} is isomorphic to the perfect tower $\{\left(k\llbracket x_1,\ldots,x_d\rrbracket/(\ol{f})\right)^\wedge_{\ol{g}}, \varphi\}_{i\geq 0}$, where $\ol{g}$ is the image of $g$ in $k\llbracket x_1,\ldots,x_d\rrbracket$. 
  \item The small tilt $I_i^{s.\flat}\subset R_i^{s.\flat}$ of $I_0^{s.\flat}$ is isomorphic to $(\ol{g})\subset \left( k\llbracket x_1,\ldots,x_d\rrbracket/(\ol{f})\right)^\wedge_{\ol{g}}$.
  \end{enumerate}
\end{proposition}

\begin{proof}
(1) We are going to check the conditions \textbf{(a)} $\sim$ \textbf{(g)} in \cref{def:inseptower,def:perfectoidtower}.

\textbf{(a)} follows from $p\in fR\subset gR$.

\textbf{(b)} Let $i\geq 0$. Since
\begin{align}
\label{eq:modg}
R_i/gR_i &\cong k\llbracket x_1,\ldots,x_d\rrbracket[x_1^{1/p^i}, \ldots,x_d^{1/p^i}]/(\ol{f}^{1/p^i},\ol{g}) \\
&\xr{\cong} k\llbracket x_1,\ldots,x_d\rrbracket/(\ol{f},\ol{g}^{p^i}), \notag
\end{align}
the induced map $\ol{t_i}\colon R_i/gR_i\to R_{i+1}/gR_{i+1}$ is isomorphic to the map
\begin{equation}
\label{eq:modgt}
k\llbracket x_1,\ldots,x_d\rrbracket/(\ol{f},\ol{g}^{p^i}) \to k\llbracket x_1,\ldots,x_d\rrbracket/(\ol{f},\ol{g}^{p^{i+1}})
\end{equation}
induced by the absolute Frobenius of $k\llbracket x_1,\ldots,x_d\rrbracket$. Thus the injectivity of $\ol{t_i}$ follows from the assumption that $k\llbracket x_1,\ldots,x_d\rrbracket/(\ol{f})$ is reduced.

\textbf{(c)} \textbf{(d)} We deduce from \eqref{eq:modg} that the $i$-th Frobenius $F_i\colon R_{i+1}/gR_{i+1}\to R_i/gR_i$ is given by the canonical surjection
\begin{equation}
\label{eq:modgF}
k\llbracket x_1,\ldots,x_d\rrbracket/(\ol{f},\ol{g}^{p^{i+1}}) \to k\llbracket x_1,\ldots,x_d\rrbracket/(\ol{f},\ol{g}^{p^i}).
\end{equation}

\textbf{(f)} The kernel of \eqref{eq:modgF} is $(\ol{f},\ol{g}^{p^i})/(\ol{f},\ol{g}^{p^{i+1}})$, which corresponds to the ideal $g^{1/p}(R_i/gR_i)$.

\textbf{(e)} follows because $R$ is a complete local ring of mixed characteristic $(0,p)$.

\textbf{(g)} follows automatically, since $R_i$ is a domain and $g\neq 0$.

(2) By the proof of (1), the first perfectoid pillar $I_1$ is $g^{1/p}(R_1/gR_1)$. For $i\geq 2$, the $i$-th perfectoid pillar $I_i\subset R_i$ is $g^{1/p^i}R_i$ by the definition of the Frobenius projection \eqref{eq:modgF}.

(3) By \eqref{eq:modg}, the $i$-th small tilt is
\[
R_i^{s.\flat}=\varprojlim_i k\llbracket x_1,\ldots,x_d\rrbracket/(\ol{f},\ol{g}^{p^n}) = \left(k\llbracket x_1,\ldots,x_d\rrbracket/(\ol{f})\right)^\wedge_{\ol{g}},
\]
the $\ol{g}$-adic comletion of $k\llbracket x_1,\ldots,x_d\rrbracket/(\ol{f})$. Moreover, the transition map $t_i^{s.\flat}\colon R_i^{s.\flat} \to R_{i+1}^{s.\flat}$ is the projective limit of the $p$-power maps \eqref{eq:modgt}, which is isomorphic to the absolute Frobenius of $\left(k\llbracket x_1,\ldots,x_d\rrbracket/(\ol{f})\right)^\wedge_{\ol{g}}$.

(4) The small tilt $I_i^{s.\flat}\subset R_i^{s.\flat}$ of $I_i$ is the kernel of the horizontal map of the commutative diagram
\[
\begin{tikzcd}
R_i^{s.\flat} \rar["\pr_0"] \dar["\cong"'] & R_i/gR_i \rar[twoheadrightarrow] \dar["\cong"] & R_i/g^{1/p^i}R_i \dar["\cong"] \\
k\llbracket x_1,\ldots,x_d\rrbracket/(\ol{f}) \rar["\pr_0"] & k\llbracket x_1,\ldots,x_d\rrbracket/(\ol{f},\ol{g}^{p^i}) \rar & k\llbracket x_1,\ldots,x_d\rrbracket/(\ol{f},\ol{g}),
\end{tikzcd}
\]
where the right two horizontal arrows are the canonical projections. This shows that $I_i^{s.\flat}\cong (\ol{g})\subset \left(k\llbracket x_1,\ldots,x_d\rrbracket/(\ol{f})\right)^\wedge_{(\ol{g})}$.
\end{proof}

\begin{example}
\label{ex:ZpTU}
  \begin{enumerate}
  \item Consider the ramified complete regular local ring $R=\Z_p\llbracket T,U\rrbracket/(p-TU) \cong \Z_p\llbracket T\rrbracket[\frac{p}{T}]$, and take $g$ to be $U=\tfrac{p}{T}$. Then the resulting preperfectoid tower is the same as that in \cref{ex:ZpT}.
  \item The ramified complete regular local ring $R=\Z_p\llbracket T,U\rrbracket/(p-T^2U^3)$ admits a preperfectoid tower arising from $(R,g)$, where $g$ is either one of $U,U^2,U^3,T,TU,TU^2,TU^3,T^2U,T^2U^2$, or $T^2U^3$ ($=p$).
  \end{enumerate}
\end{example}


\end{document}